\documentclass[12pt]{amsart}

\usepackage{amssymb}
\usepackage{amsmath}
\usepackage{amscd}
\usepackage{multicol}
\usepackage[dvipdfmx]{graphicx}

\title[Birational transformations belonging to Galois points]{Birational transformations belonging to Galois points for a certain plane quartic}
\author{Kei Miura}
\address{Department of Mathematics, National Institute of Technology, Ube College, Yamaguchi 755-8555, Japan}
\email{kmiura@ube-k.ac.jp}

\date{April 27, 2018}


\newtheorem{theorem}{Theorem}

\newtheorem{claim}{Claim}

\theoremstyle{definition}
\newtheorem{definition}{Definition}

\newtheorem{problem}{Problem}

\theoremstyle{remark}
\newtheorem{remark}{Remark}

\newenvironment{namelist}[1]{%
\begin{list}{}
  {
   \settowidth{\labelwidth}{#1}
   \setlength{\leftmargin}{2.5\labelwidth}}
}{%
\end{list}}

\def\dfrac#1#2{{\displaystyle\frac{#1}{#2}}}

\begin{document}

\begin{abstract}
In this note, we study birational transformations belonging to Galois points 
for a certain plane quartic curve. 
In fact, we see that they can be extended to Cremona transformations. 
In particular, we determine their conjugacy class
and show that they are all conjugate to linear transformations. 
\end{abstract}

\maketitle

\section{Introduction}
In this note, we study birational transformations belonging to Galois points 
for a certain plane quartic curve. 
First, we briefly recall the notion of Galois points. 
Let $k$ be the field of complex numbers ${\mathbb C}$. 
We fix it as the ground field of our discussion. 
Let $C$ be an irreducible curve in ${\mathbb P}^2$ of degree $d$ $(d \geq 3)$, 
and $k(C)$ be the function field of $C$. 
Taking a point $P$ of ${\mathbb P}^2$, we consider the projection 
$\pi_P : C \dashrightarrow {\mathbb P}^1$, 
which is the restriction of the projection 
${\mathbb P}^2 \dashrightarrow {\mathbb P}^1$ with center $P$. 
Then, we obtain the field extension induced by $\pi_P$, 
i.e., $\pi_P^* : k({\mathbb P}^1) \hookrightarrow k(C)$. 
By putting $K_P = \pi_P^*(k({\mathbb P}^1))$, we have the following definition.

\begin{definition}
The point $P$ is called a {\em Galois point} for $C$ if the field extension $k(C) / K_P$ 
is Galois. 
Then, we put $G_P = {\rm Gal} (k(C)/ K_P)$, which we call a {\em Galois group at $P$}. 
In particular, a Galois point $P$ is called a {\em smooth Galois point} 
if $P \in {\rm Reg} (C)$, where ${\rm Reg} (C)$ is the open subset of 
all non-singular points of $C$. 
\end{definition}

The notion of {\em Galois point} for $C$ was introduced by Yoshihara (cf. \cite{m-y1}) 
to study the structure of the field extension $k ( C ) / k$ from geometrical viewpoint.

\begin{definition}
We denote by $\delta (C)$ the number of smooth Galois points on $C$. 
\end{definition}

We have studied Galois points for some cases (cf. \cite{miura1}, \cite{miura4}, \cite{miura5}, 
\cite{m-y1}, \cite{yoshi}, \cite{yoshi2}, \cite{yoshi3}, etc.). 
For example, we have determined $\delta (C)$ and 
we have found the characterization of $C$ having Galois points. 
On the other hand, when $P$ is not a Galois point, 
we have determined the monodromy group and constructed the variety 
corresponding to the Galois closure.

\bigskip

The objective of this note is to investigate the action of $G_P$ on $C$. 
Specifically, our problem is stated as follows.

Suppose that $P$ is a Galois point for $C$. 
Let $\sigma$ be an element of $G_P$. 
Then, $\sigma$ induces a birational transformation of $C$ over ${\mathbb P}^1$, 
which we denote by the same letter $\sigma$. 
We call $\sigma$ a {\em birational transformation belonging to Galois point} $P$.

\begin{problem}
Study the properties of $\sigma$. 
\end{problem}

For this problem, we first state the following theorem.

\begin{theorem}[Yoshihara, \cite{yoshi}, \cite{yoshi2}]\label{yoshi}
Let $V$ be a smooth hypersurface in ${\mathbb P}^n$ of degree $d$ $(d \geq 4)$. 
Then, every element $\sigma \in G_P$ is a restriction of a projective transformation of ${\mathbb P}^n$. 
In other words, there exists $\widetilde{\sigma} \in {\rm PGL} (n + 1, k)$ such that $\widetilde{\sigma} |_V = \sigma$. 
Furthermore, $G_P$ is a cyclic group. 
\end{theorem}

The conditions of smoothness and $d \geq 4$ are both required for the theorem to hold. 
Hence, in this note, we study the problem for a singular plane quartic curve $C$. 
In particular, we study the case in which $C$ has a simple cusp of multiplicity three.

Suppose that $C$ is a quartic with a simple cusp of multiplicity three. 
Then, it is well known that $C$ is projectively equivalent 
to one of the following (cf.\ \cite{namba1}):

\begin{namelist}{(a)}
 \item[{\rm (a)}] $X^4 - X^3 Y + Y^3 Z = 0$,
 \item[{\rm (b)}] $X^4 - Y^3 Z = 0$.
\end{namelist}

In the case of (a), $C$ has two flexes of order one. 
Indeed $Q_1 = (0 : 1 : 0)$ and $Q_2 = (8 : 16 : 3)$ are flexes. 
Then, the tangent lines at these flexes meet $C$ at $P_1 = (1 : 1 : 0)$ and 
$P_2 = (8 : -16 : 3)$, respectively. 
We can easily check that $\pi_{P_i}$ is a totally ramified triple covering. 
Hence, we see that $P_i$ is a smooth Galois point for $C$ $(i= 1, 2)$.
By the condition of flexes of $C$, there is no other smooth Galois point. 
Hence, $\delta (C) = 2$. 
In the case of (b), we have a smooth Galois point $P_3 = (0 : 1 : 0)$, which is a flex of order two. 
Since there is no more flex, $\delta (C) = 1$.

\begin{remark}
The curve of type ${\rm (b)}$ has a Galois point $P_4 = (1 : 0 : 0) \in {\mathbb P}^2 \setminus C$. 
We call such a Galois point an {\em outer Galois point}. 
We can easily check that $P_4$ is a unique outer Galois point for $C$ of type ${\rm (b)}$.  
\end{remark}

The objective of this note is to investigate the properties of birational transformations beloinging to 
$P_1$, $P_2$, and $P_3$. 
We note that the Galois groups at $P_i$ are all isomorphic to the cyclic group of order three, ${\mathbb Z}_3$ 
$(i = 1, 2, 3)$. 
By setting $G_{P_i} = \langle \sigma_i \rangle$ $(i = 1, 2, 3)$, we state our main results as follows.

\begin{theorem}
All the transformations $\sigma_1$, $\sigma_2$, and $\sigma_3$ are (restrictions of) Cremona transformations. 
Furthermore, $\sigma_1$ and $\sigma_2$ are conjugate to linear transformations. 
In addition, $\sigma_3$ is a linear transformation.  
\end{theorem}

\begin{remark}
By referring to \cite{fer}, \cite{bb}, we have the classification of the conjugacy classes of elements of order three in ${\rm Bir} ({\mathbb P}^2)$ 
as follows. 
 \begin{itemize}
  \item linear transformation
  \item automorphisms of special del Pezzo surfaces of degree 3
  \item automorphisms of special del Pezzo surfaces of degree 1
 \end{itemize}
\end{remark}

\bigskip

The Cremona group ${\rm Bir} ({\mathbb P}^2)$ is the group of birational transformations of ${\mathbb P}^2$. 
This is a classical object in algebraic geometry. 
Hence, there are many results on ${\rm Bir} ({\mathbb P}^2)$, e.g., \cite{blanc}, \cite{DI}. 
However, we have obtained very few results on birational transformations belonging to Galois points, e.g.,  \cite{miura4}, \cite{miura5}, \cite{yoshi3}. 
In particular, it is difficult to determine when a birational transformation extends to a Cremona transformation.

\section{Proofs}
First, we prove the case for type ${\rm (a)}$.

\begin{claim}
Suppose that $C$ is a curve of type ${\rm (a)}$. 
Then, there exists a linear automorphism $A$ of $C$ such that $A (P_1) = P_2$.   
\end{claim}

\begin{proof}
By putting $A = \left(
\begin{array}{ccc}
16  & -8  &  0  \\
0   & -16 &  0  \\
4   & -1  &  16 \\
\end{array}
\right)$, we can easily check that $A (C) = C$ and $A (P_1) = P_2$. 
\end{proof}

Therefore, it is sufficient to study on $\sigma_1$. 
We can obtain a concrete representation of $\sigma_1$ as follows.

\begin{claim}
$\sigma_1$ is represented as $(X : Y : Z) \mapsto
(X Y : Y ((\omega - 1) X + \omega Y) : Z ((\omega - 1) X + \omega Y))$, 
up to projective coordinate change, where $\omega$ is a primitive cubic root of unity.  
\end{claim}

\begin{proof}
By taking a suitable projective coordinate, we may assume that $P_1$ is translated to $P_1' = (1 : 0 : 0)$. 
Then, $\pi_{P_1'}$ is represented as $(X : Y : Z) \dashrightarrow (Y : Z)$ and 
$C$ is defined by $(X + Y)^3 Z - X^3 Y = 0$. 
By putting $x = X / Z$, $y = Y / Z$, we can obtain the field extension induced by $\pi_{P_1'}$ as
\begin{equation*}
 \begin{CD}
  k( C ) @= k(x, y)     \\
  \cup   @. \cup \\
  k({\mathbb P}^1)@= k(y)\\
 \end{CD}
\end{equation*}
where $k(x, y) / k(y)$ is given by the equation $(x + y)^3 - x^3 y = 0$. 
Then, we obtain $\left( \dfrac{x + y}{x} \right)^3 = y$, i.e., $\left( 1 + \dfrac{y}{x} \right)^3 = y$. 
We put $G_{P_1'} = \langle \sigma_1' \rangle$. 
Hence, we infer that $\sigma_1' \left( 1 + \dfrac{y}{x} \right) =  \omega \left( 1 + \dfrac{y}{x} \right)$. 
Thus, we conclude that $\sigma_1' (x) = \dfrac{y x}{\omega x - x + \omega y}$. 
Since $\sigma_1' (x : y: 1) = (\sigma_1' (x) : y : 1)$ and $x = X / Z$, $y = Y / Z$, 
we get the claim. 
We note that $\sigma_1'$ is conjugate to $\sigma_1$ up to projective coordinate change.  
\end{proof}

By using the above representation of $\sigma_1'$, 
we can prove the following.

\begin{claim}
$\sigma_1$ is conjugate to a linear transformation. 
\end{claim}

\begin{proof}
By the above claim, $\sigma_1'$ is represented as $\sigma_1' (x) = \dfrac{y x}{\omega x - x + \omega y}$. 
We may consider $\sigma_1'$ as an element 
$$M_{\sigma_1'} = \left(
\begin{array}{cc}
y             &  0       \\
\omega - 1    & \omega y \\
\end{array}
\right) \in {\rm PGL} (2, k(y))$$
Then, by putting $P = \left(
\begin{array}{cc}
-y   &  0       \\
1    &  1       \\
\end{array}
\right)$, 
we have $$P^{-1} M_{\sigma_1'} P = \left(
\begin{array}{cc}
y   &         0       \\
0   &  \omega y       \\
\end{array}
\right).$$ 
Assuming $\widetilde{\sigma_1}$ to be a transformation defined by $P^{-1} M_{\sigma_1'} P$, 
we have $\widetilde{\sigma_1} (x) = \dfrac{y x}{\omega y} = \omega^2 x$. 
Then, we conclude that $\sigma_1$ is conjugate to $\widetilde{\sigma_1}$, which is a 
linear transformation defined by
$\left(
\begin{array}{ccc}
\omega^2  & 0  &  0  \\
0         & 1  &  0  \\
0         & 0  &  1  \\
\end{array}
\right)$. 
\end{proof}

\begin{remark}
The transformation defined by $P$ above is a Cremona transformation
$(X : Y : Z) \mapsto
(- X Y : Y ( X + Z) : Z (X + Z))$. 
\end{remark}

Thus, we prove the theorem for type ${\rm (a)}$.

\bigskip

Next, we prove the case for type ${\rm (b)}$. 
By putting $x = X / Y$, $y = Z / Y$, we obtain the field extension $\pi_{P_3}^*$ as 
$k(x, y) / k(y)$, where $x^3 - y = 0$. 
Then, it is easy to see that $\sigma_3$ is represented as 
$\left(
\begin{array}{ccc}
\omega    & 0  &  0      \\
0         & 1  &  0      \\
0         & 0  &  \omega \\
\end{array}
\right)$.

Thus, we complete the theorem.

\section{Problems}

Finally, we raise some problems.

\begin{problem}
Let $P$ be a Galois point for $C$. 

 \begin{namelist}{111}
  \item[(1)] Find the condition when the birational transformation belonging to $P$ can be extended to 
  a Cremona transformation (cf. \cite{miura5}, \cite{yoshi3}). 
  \item[(2)] Suppose that the birational transformation belonging to $P$ is extended to a Cremona transformation. 
  Then, is this a de Jonqui\`{e}res transformation?
  If this is not true, find the condition when that transformation becomes de Jonqui\`{e}res. 
  \item[(3)] Investigate the relation among Galois points, the decomposition group and the inertia group. 
  The decomposition group is the group of Cremona transformations that preserves $C$, and the inertia group is the 
  group of Creomna transformations that fix $C$. That is, 
  $${\rm Dec} (C) = \left\{ \varphi \in {\rm Bir} ({\mathbb P}^2) | \, \varphi|_{C} : C \dashrightarrow C \right\},$$
  $${\rm Ine} (C) = \left\{ \varphi \in {\rm Bir} ({\mathbb P}^2) | \, \varphi|_{C} = {\rm id}_C \right\}.$$
  Then, does the element of ${\rm Dec} (C)$ (resp.\ ${\rm Ine} (C)$) preserve the Galois point $P$ for $C$?
 \end{namelist}
\end{problem}


\begin{thebibliography}{n}


\bibitem{bb}A.\ Beauville and J.\ Blanc,
On Cremona transformations of prime order, 
{\em C.\ R.\ Math.\ Acad.\ Sci.\ Paris} {\bf 339} (2004), 257--259.


\bibitem{blanc}J.\ Blanc,
Elements and cyclic subgroups of finite order of the Cremona group, 
{\em Comment.\ Math.\ Helv.} {\bf 86} (2011), 469--497.


\bibitem{DI}I.\ Dolgachev and V.\ Iskovskikh, Finite subgroups of the plane Cremona group, 
Algebra, Arithmetic, and Geometry, {\em Progr.\ Math.} {\bf 269} (2009), 443--548. 


\bibitem{fer}T.\ de Fernex,
On planar Cremona maps of prime order, 
{\em Nagoya Math.\ J.} {\bf 174} (2004), 1--28.


\bibitem{miura1}K.\ Miura,
Field theory for function fields of singular plane quartic curves, 
{\em Bull.\ Austral.\ Math.\ Soc.} {\bf 62} (2000), 193--204.


\bibitem{miura4}\bysame, 
Galois points for plane curves and Cremona transformations, 
{\em J.\ Algebra} {\bf 320} (2008), 987--995. 


\bibitem{miura5}\bysame,
On dihedral Galois coverings arising from Lissajous's curves, 
{\em J.\ Geom.} {\bf 91} (2009), 63--72. 


\bibitem{m-y1}K.\ Miura and H.\ Yoshihara, 
Field theory for function fields of plane quartic curves, 
{\em J.\ Algebra} {\bf 226} (2000), 283--294. 


\bibitem{namba1}M.\ Namba, 
Geometry of projective algebraic curves,
Marcel Dekker, New York, Basel, 1984.


\bibitem{pansim}I.\ Pan and A.\ Simis,  
Cremona maps of de Jonqui\`{e}res type, 
{\em Canad.\ J.\ Math.} {\bf 67} (2015), 923--941.


\bibitem{yoshi}H.\ Yoshihara, 
Function field theory of plane curves by dual curves, 
{\em J.\ Algebra} {\bf 239} (2001), 340--355.


\bibitem{yoshi2}\bysame, 
Galois points for smooth hypersurfaces, 
{\em J.\ Algebra} {\bf 264} (2003), 520--534.


\bibitem{yoshi3}\bysame, 
Rational curve with Galois point and extendable Galois automorphism, 
{\em J.\ Algebra} {\bf 321} (2009), 1463--1472.

\end{thebibliography}
\end{document}